\documentclass[a4paper, 11pt]{amsart}

\usepackage[english]{babel}
\usepackage{xcolor}

\usepackage{amsmath,amssymb,amsthm}
\usepackage{amsrefs}

\newtheorem{theorem}{Theorem}[section]
\newtheorem{lemma}[theorem]{Lemma}
\newtheorem{proposition}[theorem]{Proposition}
\newtheorem{corollary}[theorem]{Corollary}

\theoremstyle{remark}
\newtheorem{remark}[theorem]{Remark}

\theoremstyle{definition}
\newtheorem{definition}[theorem]{Definition}

\newcommand{\NN}{\mathbb{N}}
\newcommand{\ZZ}{\mathbb{Z}}
\newcommand{\CC}{\mathbb{C}}
\newcommand{\po}[2]{\left\{ {#1}, {#2} \right\}}

\newcommand{\field}[1]{\mathbb{#1}}

\newcommand{\C}{\field{C}}

\newcommand{\N}{\field{N}}

\DeclareMathOperator{\tr}{Tr}

\newcommand{\matnn}{\mathrm{M}_{n}(\C)}

\newcommand{\camo}{\mathcal{C}_n}
\newcommand{\holo}{\mathop{\mathcal{O}}}
\DeclareMathOperator{\lie}{Lie}
\DeclareMathOperator{\aut}{Aut}
\newcommand{\cabove}{\widehat{\mathcal{C}}_n}
\newcommand{\rabove}{\widehat{\mathcal{S}}_n}

\newcommand{\tlcamo}{\mathcal{S}_n}
\DeclareMathOperator{\rank}{rank}
\DeclareMathOperator{\gl}{GL}
\newcommand{\RR}{\mathbb{R}}

\DeclareMathOperator{\diff}{Diff}

\DeclareMathOperator{\gln}{G}

\title{Direct Products for the Hamiltonian Density Property}
\author{Rafael B. Andrist \and Gaofeng Huang}


\dedicatory{Dedicated to Josip Globevnik on the occasion of his eightieth birthday}


\begin{document}

\begin{abstract}
We show that the direct product of two Stein manifolds with the Hamiltonian density property enjoys the Hamiltonian density property as well. We investigate the relation between the Hamiltonian density property and the symplectic density property. We then establish the Hamiltonian and the symplectic density property for $(\CC^\ast)^{2n}$ and for the so-called traceless Calogero--Moser spaces. As an application we obtain a Carleman-type approximation for Hamiltonian diffeomorphisms of a real form of the traceless Calogero--Moser space. 
\end{abstract}

\maketitle

\section{Introduction}

The first study of large holomorphic symplectic automorphism groups goes back to Forstneri\v{c} \cite{MR1408866} in 1996 who showed that the holomorphic symplectic shears generate a dense subgroup of all the holomorphic symplectic automorphisms of $\CC^{2n} \cong T^\ast \CC^{n}$ with the standard symplectic form. After a long hiatus, Deng and Wold \cite{MR4423269} in 2022 considered the holomorphic symplectic automorphism group of co-adjoint orbits of Lie groups. In 2025, the authors \cite{CaloSymplo} introduced the notion of the Hamiltonian density property (see below) as well as of the symplectic density property for Stein manifolds, and showed that the Calogero--Moser spaces $\camo$ of $n$ particles satisfy these two properties.

\smallskip

Let $X$ be a complex manifold. A \emph{holomorphic symplectic form} $\omega$ on $X$ is a closed, nowhere degenerate holomorphic $(2,0)$-form on $X$. We call holomorphic map $F \colon X \to X$ symplectic if $F^\ast \omega = \omega$. We call a holomorphic vector field $V$ on $(X, \omega)$ \emph{symplectic}, if $i_V \omega$ is a closed form; and we call $V$ \emph{Hamiltonian} if $i_V \omega$ is an exact form, i.e.\ $i_V \omega = df$ for some holomorphic function $f \colon X \to \CC$ which is called the \emph{Hamiltonian (function)} of $V$. On a connected manifold $X$, the Hamiltonian function of a given vector field is uniquely determined up to an additive constant. 
By $\mathrm{SVF}_\omega(X)$ we denote the Lie algebra of \emph{symplectic vector fields} on $X$, and by $\mathrm{HVF}_\omega(X)$ we denote the Lie algebra of \emph{Hamiltonian vector fields} on $X$. The Lie bracket is given by the commutator of vector fields. 

The Poisson bracket $\{f,g\}$ between Hamiltonian functions $f, g$ is related to the Lie bracket between the corresponding Hamiltonian vector fields $V_f, V_g$ as follows
\begin{align} \label{rel: Poisson-Lie}
    d \{ f, g \} = - i_{[V_f, V_g]} \omega
\end{align}

Denote by $\flat \colon T X \to T^\ast X,  V \mapsto i_V \omega$ the canonical isomorphism induced by the (nondegenerate) holomorphic symplectic form $\omega$, and by $\sharp \colon T^\ast X \to T X$ the inverse of $\flat$. For a vector field $V$ we use the notation $V^\flat = \flat (V)$ for its associated 1-form and for a 1-form $\alpha$, $\alpha^\sharp =\sharp(\alpha)$ for its associated vector field. 

Then 
\begin{align} \label{iso-Theta}
    \Theta \colon \holo(X)/\CC \to \mathrm{HVF}_\omega(X), f \mapsto \flat( df)
\end{align}
is a Lie algebra isomorphism between holomorphic functions on $X$ modulo the constants, and the Hamiltonian vector fields on $X$. 

\begin{definition}
\begin{enumerate} \hfill
    \item Let $(X,\omega)$ be a smooth affine complex variety with an algebraic symplectic form. We say that $(X, \omega)$ enjoys the \emph{algebraic Hamiltonian density property} if the Lie algebra generated by the complete algebraic Hamiltonian vector fields on $X$ equals $\mathrm{HVF}^{\mathrm{alg}}_\omega(X)$. 
    \item Let $(X,\omega)$ be a Stein manifold with a holomorphic symplectic form. We say that $(X, \omega)$ enjoys the \emph{Hamiltonian density property} if the Lie algebra generated by the complete holomorphic Hamiltonian vector fields on $X$ is dense in $\mathrm{HVF}_\omega(X)$.    \item Let $(X,\omega)$ be a Stein manifold with a holomorphic symplectic form. We say that $(X, \omega)$ enjoys the \emph{symplectic density property} if the Lie algebra generated by the complete holomorphic symplectic vector fields on $X$ is dense in $\mathrm{SVF}_\omega(X)$.
\end{enumerate}
\end{definition}

We denote the first holomorphic de Rham cohomology group on a complex manifold $X$ by
\[
H^1_{\mathrm{dR}}(X) = \{ \omega \,:\, \text{holomorphic $1$-form on } X, d \omega = 0 \} / \{ df \,:\, f \in \holo(X) \}
\]

\begin{remark} \hfill
\begin{enumerate}
    \item For establishing the Hamiltonian density property of $X$, it suffices to show that the Lie algebra generated by a set of functions inducing complete Hamiltonian vector fields, is dense in $\holo(X)/\CC$. 
    \item For a smooth affine algebraic variety $X$, the algebraic Hamiltonian density property implies the Hamiltonian density property since every holomorphic function on $X$ can be approximated by polynomials on $X$.
    \item If $X$ is a Stein manifold with vanishing first holomorphic de Rham cohomology group $H^1_{\mathrm{dR}}(X)$, then every holomorphic symplectic vector field is in fact Hamiltonian, and the notions of the Hamiltonian density property and the symplectic density property coincide.
    \item If $X$ is a Stein manifold, then $H^1_{\mathrm{dR}}(X)$ is isomorphic to the singular cohomology group $H^1(X, \CC)$ which is a topological invariant.
\end{enumerate}
\end{remark}

In this article we develop the general theory further and revisit again the Calogero--Moser spaces. 
\begin{enumerate}
    \item We prove that direct products of manifolds with the Hamiltonian density property have the Hamiltonian density property, see Section \ref{sec-product}.
    \item We give sufficient criteria to deduce the symplectic density property from the Hamiltonian density property even if the first holomorphic de Rham cohomology is non-trivial, see also Section \ref{sec-product}.
    \item As an application, we establish the Hamiltonian and the symplectic density property for $\CC \times \CC^\ast$,  $(\CC^\ast)^2$ and their products, see Section \ref{sec-applications}.
    \item We consider a special form of the Calogero--Moser space which we call \emph{traceless Calogero--Moser space} $\tlcamo$ and establish the Hamiltonian and the symplectic density property for $\tlcamo$, see Section \ref{sec-traceless}. As a consequence, we also obtain a Carleman-type approximation for Hamiltonian diffeomorphisms of a real form of $\tlcamo$.
\end{enumerate}

\section{Direct Products}
\label{sec-product}

Denote by $p_X \colon X \times Y \to X$ and $p_Y \colon X \times Y \to Y$ the canonical projections to the factors, and by $i_X \colon X \to X \times Y$ and $i_Y \colon Y \to X \times Y$ the inclusions. Given symplectic forms $\omega_X$ and $\omega_Y$ on $X$ and $Y$, respectively, it is well-known that we obtain a symplectic form $\omega_{X \times Y} = p_X^\ast \omega_X + p_Y^\ast \omega_Y$ on $X \times Y$.

\begin{theorem}
\label{thm:directprodHam}
Let $(X, \omega_X)$ and $(Y, \omega_Y)$ be smooth affine varieties equipped with algebraic symplectic forms. If $(X, \omega_X)$ and $(Y, \omega_Y)$ have the algebraic Hamiltonian density property, then $(X \times Y, \omega_{X \times Y})$ has the algebraic Hamiltonian density property.
\end{theorem}

\begin{theorem}
\label{thm:directprodHamStein}
    Let $(X, \omega_X)$ and $(Y, \omega_Y)$ be Stein manifolds equipped with holomorphic symplectic forms. If $(X, \omega_X)$ and $(Y, \omega_Y)$ have the Hamiltonian density property, then $(X \times Y, \omega_{X \times Y})$ has the Hamiltonian density property.
\end{theorem}

These two theorems will follow from the following key proposition.

\begin{proposition}
\label{prop:tensor}
Let $(X, \omega_X)$ and $(Y, \omega_Y)$ be smooth affine varieties equipped with algebraic symplectic forms or Stein manifolds equipped with holomorphic symplectic forms. 

Let $\mathfrak{h}_X$ and $\mathfrak{h}_Y$ be some Lie subalgebras of Hamiltonian vector fields on $X$ and $Y$, respectively. Assume further that $\mathfrak{h}_X$ and $\mathfrak{h}_X$ are generated by complete Hamiltonian vector fields with corresponding Hamiltonian functions $(f_k)_k$ and $(g_\ell)_\ell$, respectively. 
Then $\mathfrak{h}_X \otimes \mathfrak{h}_Y$ on $X \times Y$ is generated by the complete Hamiltonian vector fields $(f_k)_k \cup (g_\ell)_\ell \cup (f_k^2)_k \cup (g_\ell^2)_\ell \cup ((f_k + g_\ell)^2)_{k,\ell}$.
\end{proposition}

\begin{proof}[Proof of Theorem \ref{thm:directprodHam}]
Apply Proposition \ref{prop:tensor} with $\mathfrak{h}_X = \CC[X]$ and $\mathfrak{h}_Y = \CC[Y]$ and observe that the natural embedding $\CC[X] \otimes \CC[Y] \hookrightarrow \CC[X \times Y]$ is a bijection.
\end{proof}

\begin{proof}[Proof of Theorem \ref{thm:directprodHamStein}]
Apply Proposition \ref{prop:tensor} with $\mathfrak{h}_X = \holo(X)$ and $\mathfrak{h}_Y = \holo(Y)$. We need to take the closure of the tensor product $\holo(X) \otimes \holo(Y) \hookrightarrow \holo(X \times Y)$: The space of holomorphic functions on a Stein manifold is nuclear, see e.g.\ the textbook of Trèves \cite{Treves}, Chapters 50 and 51. The space of holomorphic functions on the direct product is the topological tensor product of the corresponding spaces of holomorphic functions:
\[
\holo(X \times Y) = \holo(X) \,\widehat{\otimes}\, \holo(Y)
\]
The proof of this fact is analogous to the one of Theorem 51.6 in \cite{Treves}.
\end{proof}

\begin{lemma}
\label{lemma:poissonprod}
Let $(X, \omega_X)$ and $(Y, \omega_Y)$ be smooth affine varieties equipped with algebraic symplectic forms. Let $\po{.}{.}_X$ and $\po{.}{.}_Y$ denote the corresponding Poisson brackets. Then the Poisson bracket $\po{.}{.}_{X \times Y}$ on $(X \times Y, \omega_{X \times Y})$ is given and completely determined by 
\begin{align*}
\forall f_1, f_2 \in \CC[X] \,:\, \po{f_1}{f_2}_{X \times Y} &= \po{f_1}{f_2}_{X} \\
\forall g_1, g_2 \in \CC[Y] \,:\, \po{g_1}{g_2}_{X \times Y} &= \po{g_1}{g_2}_{Y} \\
\forall f_1 \in \CC[X], g_1 \in \CC[Y] \,:\, \po{f_1}{g_1}_{X \times Y} &= 0
\end{align*}
where we understand a function defined on one factor as a function on the product that is independent on the other factor.
\end{lemma}

\begin{proof}
Let $V_1$ and $V_2$ denote the vector fields on $X$ corresponding to the Hamiltonian functions $f_1$ and $f_2$, respectively. Denote by $\widetilde{V_1}$ and $\widetilde{V_2}$ the trivial extensions of $V_1$ and $V_2$, respectively, to $X \times Y$, i.e.\ 
\[
\widetilde{V_j}_{(x,y)} = (i_X)_\ast (V_j)_x
\]
By $\widetilde{f_j}$ we denote the extension of the Hamiltonian functions, $\widetilde{f_j} = f_j \circ p_X$. 
We have that
\[
d \widetilde{f_j} = d f_j \circ d p_X = d f_j \circ (p_X)_\ast = i_{\omega_X}V_j \circ (p_X)_\ast = i_{\omega_{X \times Y}} \widetilde{V_j}
\]
Similarly, we proceed for Hamiltonian functions $g_1$ and $g_2$ on $Y$.

Finally, we obtain $\po{f_1}{g_1}_{X \times Y} = c$ for a constant $c \in \CC$ since the two corresponding vector fields $\widetilde{V_1}$ and $\widetilde{W_1}$, respectively, commute. This can be seen immediately by a computation in local coordinates: Let $\widetilde{V_1} = \sum_k a_k(x) \frac{\partial}{\partial x_k}$ and $\widetilde{W_1} = \sum_\ell b_\ell(y) \frac{\partial}{\partial y_\ell}$,
\[
[\widetilde{V_1}, \widetilde{W_1}] = \sum_{k, \ell} \left[a_k(x) \frac{\partial}{\partial x_k}, b_\ell(y) \frac{\partial}{\partial y_\ell}\right] = 0
\]
The constant $c$ must be zero, since
$\po{f_1^2}{g_1}_{X \times Y} = 2 f_1 c$ 
corresponds to the vector field $[2 f_1 \widetilde{V_1}, \widetilde{W_1}] = 2 f_1 [\widetilde{V_1}, \widetilde{W_1}] - 2 \widetilde{W_1}(f_1) \widetilde{V_1} = 0$.
\end{proof}

\begin{remark}
\label{lemma:shear}
Let $V$ be a symplectic vector field and $f$ be a holomorphic function such that $f V$ is again symplectic. Then $f \in \ker V$ which follows from a computation in local Darboux coordinates. 
\end{remark}

\begin{lemma}
\label{lemma:sums}
Let $f \in \CC[X]$ and $g \in \CC[Y]$ be Hamiltonian functions of complete algebraic vector fields on $X$ and $Y$, respectively. Then $(f + g)^n$ is the Hamiltonian function of a complete algebraic vector field on $X \times Y$ for any $n \in \NN$.
\end{lemma}
\begin{proof}
\begin{enumerate}
\item Let $V$ denote the complete vector field on $X$ with flow $\varphi_t$ and with Hamiltonian function $f$. Let $W$ denote the complete vector field with flow $\psi_t$ on $Y$ and with Hamiltonian function $g$. We denote the trivial extension of these functions and vector fields to the product $X \times Y$ by the same letters. The flow maps for $V$ and $W$ on $X \times Y$ are given by $\Phi_t(x,y) = (\varphi_t(x),y)$ and $\Psi_t(x,y) = (x,\psi_t(y))$, respectively. The Hamiltonian function $f+g$ corresponds to the vector field $V+W$ whose flow is given by $\Phi_t \circ \Psi_t (x,y) = \Psi_t \circ \Phi_t (x,y) = (\varphi_t(x), \psi_t(y))$. 
\item Since $d h^n = n h^{n-1} dh$, the vector field corresponding to $h$ is just multiplied by the function $n h^{n-1}$. Since the resulting vector field is again Hamiltonian, this function lies in the kernel of the vector field by Remark \ref{lemma:shear} and is thus again complete. \qedhere
\end{enumerate}

\end{proof}

\begin{proof}[Proof of Proposition \ref{prop:tensor}]
Let $\mathcal{F}$ denote the Lie subalgebra of
\[
    (\CC[{X \times Y}],\po{.}{.}_{X \times Y})
\]
that is generated by Hamiltonian functions corresponding to complete algebraic vector fields.

\begin{enumerate}
\item We have $(f_k)_k \subset \mathcal{F}$ and $(g_\ell)_\ell \subset \mathcal{F}$.
\item We have $\forall k \, \forall \ell \,:\, f_k \cdot g_\ell \in \mathcal{F}$, since $2 \cdot f_k \cdot g_\ell = (f_k + g_\ell)^2 - f_k^2 - g_\ell^2$ is a linear combination of Hamiltonian functions corresponding to complete vector fields by Lemma \ref{lemma:sums}.
\item For every $g \in \CC[Y]$ we have $\forall k \,:\, f_k \cdot g \in \mathcal{F}$. 
Indeed, every element of $\CC[Y]$ can be written as a linear combination of finitely many nested Poisson brackets. Without loss of generality, we may consider only one summand of nested Poisson brackets. Note that the case of zero Poisson brackets is already covered by the previous step. 
\[
g = \po{g_1}{ \dots \po{g_{m-1}}{g_m}_Y}_Y = \po{g_1}{\tilde{g}}_Y
.\]
We now obtain
\begin{align*}
f_k \cdot g = f_k \cdot \po{g_1}{\tilde{g}}_{X,Y} &= \po{f_k g_1}{\tilde{g}}_{X,Y} \\ &= \po{f_k g_1}{ \dots \po{g_{m-1}}{g_m}_{X \times Y}}_{X \times Y} \in \mathcal{F}
\end{align*}
\item For every $f \in \CC[X]$ and every $g \in \CC[Y]$ we have $f \cdot g \in \mathcal{F}$.
Similarly, we consider nested Poisson brackets of generators in $\CC[X]$:
\[
f = \po{f_1}{ \dots \po{f_{n-1}}{f_n}_X}_X = \po{f_1}{\tilde{f}}_X
\] and obtain
\begin{align*}
f \cdot g = \po{g f_1}{\tilde{f}}_{X \times Y} = \po{g f_1}{ \dots \po{f_{n-1}}{f_n}_{X \times Y} }_{X \times Y}  \in \mathcal{F}
\end{align*}
\item Taking linear combinations of products of the form $f \cdot g$ concludes the proof. \qedhere
\end{enumerate}
\end{proof}

\begin{remark}
Let $f \in \CC[X]$ and $g \in \CC[Y]$ be Hamiltonian functions of LNDs $V$ and $W$, respectively. Then $f \cdot g \in \CC[X \times Y]$ is the Hamiltonian function of the LND $g V + f W$. To see this, just note that $f, g \in \ker V \cap \ker W$, $[V,W] = 0$, and hence
\[
(g V + f W)^n = \sum_{k=0}^n g^k f^{n-k} V^k W^{n-k}
\]
eventually vanishes on each polynomial in $\CC[X \times Y]$.
\end{remark}

\begin{proposition} \label{prop: HDPtoSDP}
    Let $X$ be a Stein manifold equipped with a holomorphic symplectic form $\omega$. Suppose that 
    \begin{enumerate}
        \item[(i)] $X$ has the Hamiltonian density property, and
        \item[(ii)] there exist complete symplectic vector fields $\{ V_j \}_J$ where $J$ is an index set, such that $\{ [V_j^\flat] \}_J$ form a basis of $H^1_{\mathrm{dR}}(X)$.
    \end{enumerate}
    Then $X$ has the symplectic density property. 
\end{proposition}
\begin{proof}
    Let $V$ be a symplectic vector field on $X$. By (ii) there are complete symplectic vector fields $V_{i_1}, \dots, V_{i_m} \in \{ V_j \}_J$ and a Hamiltonian vector field $W$ such that 
    \[
        V = c_1 V_{i_1} + \dots + c_m V_{i_m} + W
    \]
    for some $c_j \in \CC, j= 1, \dots, m$.
    By (i) the vector field $W$ can be written as a Lie combination of complete Hamiltonian vector fields. Hence $V$ is in the Lie algebra generated by complete symplectic vector fields. 
\end{proof}

\begin{proposition}
    Let $X$ be a Stein manifold equipped with a holomorphic symplectic form $\omega$. Suppose that $X$ has the symplectic density property and the dimension $m=\dim H^1_{\mathrm{dR}}(X)$ is finite.
    Then there exist complete symplectic vector fields $V_1, \dots, V_m$ such that $[V_1^\flat], \dots, [V_m^\flat]$ form a basis of $H^1_{\mathrm{dR}}(X)$. 
\end{proposition}
\begin{proof}
    Let $[\alpha_1], \dots, [\alpha_m]$ be a basis of $H^1_{\mathrm{dR}}(X)$. By the symplectic density property, each $\alpha_j^\sharp, j \in \{1, \dots, m\}$, can be approximated by a Lie combination of complete symplectic vector fields. Since every bracket of two symplectic vector fields is Hamiltonian (see e.g.\ \cite{MR1853077}*{Proposition 18.3}), there are complete symplectic vector fields $V_{j1}, \dots, V_{j k_j}$ such that $\alpha_j^\sharp$ can be approximated by $V_{j1} + V_{j2} + \dots + V_{j k_j}$
    modulo $\mathrm{HVF}_\omega(X)$ with $m_j \le m$. Among the $\sum_j m_j$ complete symplectic vector fields
    \[
        V_{11}, \dots, V_{1 k_1}, V_{21}, \dots, V_{2 k_2}, \dots, V_{m1}, \dots, V_{m k_m},
    \]
    choose a minimal set of complete symplectic vector fields $ W_{1}, \dots, W_{ \tilde{m}}$
    such that $[W_{1}^\flat], \dots, [W_{ \tilde{m}}^\flat]$
    is a minimal spanning set of 
    \[
        \mathrm{Span}([V_{11}^\flat], \dots, [V_{1 k_1}^\flat], [V_{21}^\flat], \dots, [V_{2 k_2}^\flat], \dots, [V_{m1}^\flat], \dots, [V_{m k_2}^\flat]).
    \]
    Since $[\alpha_1], \dots, [\alpha_m]$ are in this span, $\tilde{m} = m$. We thus obtain $m$ complete symplectic fields whose corresponding $1$-forms form a basis of $H^1_{\mathrm{dR}}(X)$.
\end{proof}

\section{Applications}
\label{sec-applications}

\subsection{$\CC^2$}
We equip $\CC^2$ with the standard algebraic symplectic form $\omega = d z \wedge d w$ where $(z,w) \in \CC^2$.
\begin{lemma}
    The Euclidean space $\CC^2$ equipped with the algebraic symplectic form $dz \wedge d w$ has the algebraic Hamiltonian density property. 
\end{lemma}
\begin{proof}
    The Poisson bracket is determined by $\{ z, w \} = 1$, linearity and Leibniz' rule. 
    Any monomial in $(z, w) \in \CC^2$ is in the Lie algebra generated by Hamiltonian functions corresponding to complete vector fields, since
    \begin{align*}
        \{ z^j, w^k \} = jk z^{j-1} w^{k-1}
    \end{align*}
    and $z^j \frac{\partial}{\partial w}, w^k \frac{\partial}{\partial z}$ are complete. 
\end{proof}

\subsection{$\CC \times \CC^\ast$}
We equip $\CC \times \CC^\ast$ with the unique invariant algebraic symplectic form $\omega = dz \wedge \frac{d w}{w}$ where $(z,w) \in \CC \times \CC^\ast$.

\begin{lemma}
    The smooth affine variety $\CC \times \CC^\ast$ equipped with the algebraic symplectic form $dz \wedge \frac{d w}{w}$ has the algebraic Hamiltonian density property. 
\end{lemma}
\begin{proof}
    The complete Hamiltonian vector field $V = - w \frac{\partial}{\partial w}$ corresponds to the function $f = z$ and 
    the complete Hamiltonian vector field $W =  w \frac{\partial}{\partial z}$ corresponds to the function $f = w$. Using \eqref{rel: Poisson-Lie} we obtain from the Lie bracket of these two vector fields
    \[
        [V, W] =  - w \frac{\partial}{\partial z},
    \]
the Poisson bracket $\{ z, w \} = w$.
By Leibniz' rule, $\{ z^j, w^k \} = j k z^{j-1} w^k$
for $j \in \NN, k \in \ZZ \setminus \{0 \}$. Therefore, the Lie subalgebra generated by 
\begin{align*}
    \{ z^j, w^k: j \in \NN, k \in \ZZ \setminus \{0 \} \}
\end{align*}
contains any function in $\CC[z, w, w^{-1}]$, i.e.\ any regular function on $\CC \times \CC^\ast$. 
\end{proof}

\subsection{The torus $(\CC^\ast)^{2n}$}
We equip $(\CC^\ast)^{2n}$ with the algebraic symplectic form $\omega = \sum_{j=1}^n \frac{d z_j}{z_j} \wedge \frac{d w_j}{w_j}$
where $(z_1, \dots, z_n, w_1, \dots, w_n) \in (\CC^\ast)^{2n}$.

\begin{proposition} \label{prop: torusHDP}
    The smooth affine variety $(\CC^\ast)^{2n}$ equipped with the algebraic symplectic form $\sum_{j=1}^n \frac{d z_j}{z_j} \wedge \frac{d w_j}{w_j}$ has the algebraic Hamiltonian density property. 
\end{proposition}
\begin{proof}
    It suffices to show the algebraic Hamiltonian density property for $\CC^\ast_z \times \CC^\ast_w$ with the invariant form $\frac{d z}{z} \wedge \frac{d w}{w}$ by Theorem \ref{thm:directprodHam}. 

    The complete Hamiltonian vector field
\[
    V = - j z^j w \frac{\partial}{\partial w}, \quad j \in \ZZ \setminus \{0 \}
\]
corresponds to the function $f = z^j$ and the complete Hamiltonian vector field
\[
    W =  k z w^k \frac{\partial}{\partial z}, \quad k \in \ZZ \setminus \{0 \}
\]
corresponds to the function $f = w^k$. 

Using \eqref{rel: Poisson-Lie} we obtain from the Lie bracket of these two vector fields
\[
    [V, W] = j^2 k z^j w^{k+1} \frac{\partial}{\partial w} - j k^2 z^{j+1} w^k \frac{\partial}{\partial z}
\]
the Poisson bracket $\{ z^j, w^k \} = j k z^j w^k$ 
for $j, k \in \ZZ \setminus \{0 \}$. Therefore, the Lie subalgebra generated by $ \{ z^j, w^k: j, k \in \ZZ \setminus \{0 \} \}$
contains any function in $\CC[z, w, z^{-1}, w^{-1}]$, i.e.\ any regular function on $\CC^\ast \times \CC^\ast$. 
\end{proof}

\begin{theorem}
    The smooth affine variety $(\CC^\ast)^{2n}$ equipped with the algebraic symplectic form $\sum_{j=1}^n \frac{d z_j}{z_j} \wedge \frac{d w_j}{w_j}$ has the algebraic symplectic density property.
\end{theorem}
\begin{proof}
    Since $(\CC^\ast)^{2n}$ is affine, we have
    \begin{align*}
        H^1_{dR}((\CC^\ast)^{2n})  \cong  H^1((\CC^\ast)^{2n}, \C)  = \CC^{2n}
    \end{align*}
    By Proposition \ref{prop: torusHDP} and Proposition \ref{prop: HDPtoSDP}, it suffices to find complete symplectic vector fields whose corresponding 1-forms form a basis of $H^1_{dR}((\CC^\ast)^{2n})$. The following $2n$ vector fields
    \begin{align*}
        z_1 \frac{\partial}{\partial z_1}, \dots, z_n \frac{\partial}{\partial z_n}, w_1 \frac{\partial}{\partial w_1}, \dots, w_n \frac{\partial}{\partial w_n},
    \end{align*} 
    form such a set of representatives. 
\end{proof}

\section{Traceless Calogero--Moser space}
\label{sec-traceless}
We recall the definitionof the Calogero--Moser variety by Wilson \cite{MR1626461} 
\[
\cabove := \{ (X,Y) \in \matnn \oplus \matnn: \mathrm{rank}([X,Y] - i I_n) =1 \} 
\]
and the Calogero--Moser space $\camo = \cabove /\!/ \gl_n(\C)$ of $n$ particles. Notice that we choose the constant $-i$ in the rank condition, such that this equation is preserved under the involution $\tau$ which we will introduce below in equation \eqref{traceless-involution}. In general, our computations in the current section remain valid for generic rank condition $\rank([X, Y] + \lambda I_n) = 1, \lambda \in \C \setminus \{0\}$.
Consider the decomposition 
\begin{align*}
    \cabove = \C^2 \oplus \rabove, \quad \rabove := \{ (A,B) \in \mathfrak{sl}_{n} \oplus \mathfrak{sl}_{n}: \mathrm{rank}([A,B] - iI_n) =1 \} 
\end{align*}
where the coordinate functions on the first factor $\C^2$ correspond to $(\tr X, \tr Y)$. The $\gl_n(\C)$-action on $\C^2$ is trivial; in other words, it commutes with the action of the two one-parameter unipotent groups 
\begin{align*}
    ( X, Y) \mapsto (X, Y + t I_n), \quad (X,Y) \mapsto (X+ t I_n, Y)
\end{align*}
whose Hamiltonians functions are $\tr X$ and $\tr Y$, respectively.
Hence, the decomposition descends to the $\gl_n(\C)$-quotient 
\begin{align*}
    \camo = \C^2 \oplus \tlcamo 
\end{align*}
where $\tlcamo$ denotes the quotient $\rabove /\!/ \gl_n(\C)$. 
Observe that the subvariety $\tlcamo \cong \camo /\!/ \C^2_+$ is smooth, since each of the two $\C_+$-actions changes the conjugacy class of the pair $(X,Y)$ and is thus free. 

From the symplectic point of view, consider the standard symplectic form on $\matnn \times \matnn$ given by
\begin{equation}
\widehat{\Omega} = \tr (d X \wedge d Y) = \sum_{j, k=1}^n d X_{jk} \wedge d Y_{kj}
\end{equation}
The form $\widehat{\Omega}$ is invariant under conjugation. The $\gl_n(\CC)$-action admits a moment map 
\begin{align*}
	\mu \colon \matnn \times \matnn \to \mathfrak{sl}_n (\mathbb{C}) \cong \mathfrak{sl}_n^\ast (\mathbb{C}), \; (X,Y) \mapsto [X, Y]
\end{align*}
where we identify the Lie algebra with its dual using the trace form $\left< M, N\right> =  \tr(MN)$. 

Take $\xi \in \mathfrak{sl}_n$ with diagonal entries $\xi_{jj}=0$ and with off-diagonal entries $\xi_{jk} = -i$ where $j \neq k$. The Calogero--Moser space $\camo$ can be endowed with a symplectic form $\Omega$ as the symplectic reduction $\mu^{-1}(\gln \cdot \xi)/\gln$. For the traceless Calogero--Moser space $\tlcamo$, take the restricted moment map
\begin{align*}
	\mu \colon \mathfrak{sl}_n \oplus \mathfrak{sl}_n \to \mathfrak{sl}_n (\mathbb{C}) = \mathfrak{sl}_n^\ast (\mathbb{C}), \; (A,B) \mapsto [A, B].
\end{align*}
Then $\tlcamo$ can be seen as the symplectic quotient $\mu^{-1}(\gln \cdot \xi) / \gln$. 

Let us denote the inclusion by $\jmath \colon \tlcamo \hookrightarrow \camo$. Consider the 2-form on $\tlcamo$ given by $\omega = \jmath^\ast \Omega$. 
Since the bracket $\{ \tr X, \tr Y \} = n$ between the two defining functions of the subvariety $\tlcamo$ is nowhere vanishing, the intersection of the tangent bundle $T \tlcamo$ and the orthogonal bundle $(T \tlcamo)^\perp$ is the zero section. Hence the subvariety $\tlcamo$ is symplectic in $(\camo, \Omega)$, and $\omega$ is a symplectic form on $\tlcamo$. In particular, the Poisson bracket on $(\tlcamo, \omega)$ is the restriction of the Poisson bracket on $(\camo, \Omega)$. 

Furthermore, the subvariety $\tlcamo$ is homotopy equivalent to the Calogero--Moser space $\camo$, by the following deformation retract from $\camo$ to $\tlcamo$:
\begin{align*}
    (X, Y) \mapsto \left(X - s \frac{\tr X}{n} I_n, Y - s \frac{\tr Y}{n} I_n \right), \quad s \in [0,1] 
\end{align*}
which induces isomorphisms on the level of cohomology groups. In particular, the first holomorphic de Rham cohomology group of $\tlcamo$ vanishes as it does for $\camo$, and every holomorphic symplectic vector field on $\tlcamo$ is  Hamiltonian. 

\begin{theorem} 
    The traceless Calogero--Moser space $\tlcamo$ endowed with the holomorphic symplectic form $\omega$ has the symplectic density property. 
\end{theorem}

The proof will be presented in Section \ref{proof-HDP}. 

We can then also recover the main result of \cite{CaloSymplo}:
\begin{corollary}
    The Calogero--Moser space $\camo$ has the symplectic density property.
\end{corollary}
\begin{proof}
    We have that $\camo \cong \tlcamo \times T^\ast \CC$. All spaces involved have trivial first holomorphic de Rham cohomology group.
\end{proof}

\begin{theorem}
    The identity component of the group of holomorphic symplectic automorphisms of $\tlcamo, n\ge 2$, is the closure (in the topology of uniform convergence on compacts) of the group generated by the following four families ($t \in \CC$) of holomorphic symplectic automorphisms
    \begin{align*}
        &(A, B) \mapsto (A, B - 2 t A) \\
        &(A, B) \mapsto (A + 2 t B, B) \\
        &(A, B) \mapsto (A, B - 3 t A^2 + 3 t \frac{(\tr A)^2}{n} I_n) \\
        &(A, B) \mapsto ( A \exp(2t \tr AB), B \exp(-2t \tr AB) ).
    \end{align*}
\end{theorem}
\begin{proof}
These four families are the flows corresponding to the Hamiltonian functions $\tr A^2, \tr B^2, \tr A^3$, and $(\tr AB)^2$, respectively, which generate the Poisson--Lie algebra on $\tlcamo$. By the Anders{\'e}n--Lempert theorem, see \cite{CaloSymplo}*{Theorem 12}, the flows of the generators of the Lie algebra generate a dense subgroup of the identity component of the automorphism group. 
\end{proof}

Another application of the symplectic density property of the traceless Calogero--Moser space is a holomorphic approximation for symplectic diffeomorphisms as in \cite{CaloCarleman} by the second author. For this, we choose a totally real submanifold of  $\tlcamo$ by taking the following antiholomorphic involution 
\begin{equation}
\label{traceless-involution}
	\tau (A, B) = ( A^*,  B^* )
\end{equation}
whose fixed-point set is
\begin{align*}
    \tlcamo^\tau = \{ [(A, B)] \in \tlcamo : A^* = A, B^* = B \}.
\end{align*}
Denote by $j$ the inclusion map $\tlcamo^\tau \hookrightarrow \tlcamo$. 
Pulling-back the holomorphic symplectic form $\omega$ on $\tlcamo$ by $j$ gives a real symplectic form $\omega_\RR:= j^* \omega$ on the real part $\tlcamo^\tau$. 

\begin{theorem}
    Let $k \in \NN$ and $\varphi \in \diff_{\omega_\RR}(\tlcamo^\tau)$ be a symplectic diffeomorphism which is smoothly isotopic to the identity. Then for any positive continuous function $\epsilon$ on $\tlcamo^\tau$, there exists a holomorphic symplectic automorphism $\Phi \in \aut_{\omega}(\tlcamo)$ such that 
	\[	
            \Phi(\tlcamo^\tau) = \tlcamo^\tau \quad \text{ and } \quad \| \Phi - \varphi \|_{C^k(p)} < \epsilon (p)	\text{ for all } p \in \tlcamo^\tau, 
        \]
    where $\| \cdot \|_{C^k(p)}$ is a pointwise $C^k$-seminorm on $\tlcamo^\tau$. 
\end{theorem}
\begin{proof}
    Later in Section \ref{proof-HDP}, we will see that the Lie algebra generated by the real holomorphic functions $\tr A^2, \tr B^2, \tr A^3$, and $(\tr AB)^2$ coincides with $\CC[\tlcamo]$.  Then the claim follows from a verbatim application of the proof for the Calogero--Moser space $\camo$ in \cite{CaloCarleman}. 
\end{proof}

\section{Poisson brackets for $\tlcamo$}

The following is mentioned in \cite{CM4}; it is an improvement in the number of $\CC$-algebra generators over those in Etingof and Ginzburg \cite{MR1881922}*{Section 11, p.~322, Remark (ii)} and in \cite{CaloSymplo}. 
\begin{lemma} \label{lem:finiRingGen}
    The coordinate ring $\CC [\camo] \cong \C [\cabove]^{\gl_n(\C)}$ can be generated by 
    \[
        \mathcal{G} = \{ \tr X^j Y^k : 0 \le j, k \le n \}
    \]
\end{lemma}
\begin{proof}
    By virtue of the rank condition $\rank([X,Y] - i I_n) = 1$, it is known that $\C [\camo]$ can be generated by $\{ \tr X^j Y^k :  j + k \le n^2 \}$. By the Cayley--Hamilton theorem, $\tr X^j Y^k$ with $j>n$ or $k > n$ is in the $\C$-algebra generated by $\mathcal{G}$: 
    
    For $X$ the following identity holds:
    \begin{align} \label{CaleyHam}
        X^n - \sigma_{n-1}(X) X^{n-1} + \cdots + (-1)^{n-1} \sigma_1(X) X + (-1)^n \sigma_0(X)  = 0
    \end{align}
    Multiplying \eqref{CaleyHam} with $XY^k, k \le n$ and taking the trace shows that $\tr X^{n+1}Y^k$ is in $\langle \mathcal{G} \rangle$. Suppose $\tr X^{n+j}Y^k$ is in $\langle \mathcal{G} \rangle$ for any $j \le l$, then multiplying \eqref{CaleyHam} with $X^{l+1}Y^k$ and taking the trace yields $\tr X^{n+l+1}Y^k \in \langle \mathcal{G} \rangle$. 
\end{proof}

Since $\tlcamo \hookrightarrow \camo$ is a $\gl_n(\CC)$-invariant subvariety, the restrictions of generators of $\CC[\camo]$ generate the ring of regular functions $\CC[\tlcamo]$ as $\CC$-algebra. Hence by Lemma \ref{lem:finiRingGen} 
\begin{align} 
    \CC[\tlcamo] = \langle \jmath^\ast \mathcal{G} \rangle, \quad \jmath^\ast \mathcal{G} = \{ \tr A^j B^k : 0 \le j, k \le n \}
\end{align}

Since $\omega = \jmath^\ast \Omega$ and restricting commutes with Poisson bracketing, we can work with the $\Omega$-Poisson bracket between invariant polynomials of the form $\tr A^j B^k$ on $\camo$. 

\begin{definition}[\cite{CaloSymplo}] \label{def: degrees}
A \emph{matrix monomial} is a map $M \colon \matnn \times \matnn \to \matnn$ of the form 
\[
(X,Y) \mapsto X^{p_1} Y^{q_1} X^{p_2} Y^{q_2} \cdots X^{p_m} Y^{q_m}
\] with $p_1, \dots, p_m, q_1, \dots, q_m \in \N_0$.

The \emph{bidegree} $(i,j)$ of $M$ is given by 
\[
(i,j) = (p_1 + p_2 + \dots + p_m, q_1 + q_2 + \dots + q_m)
.\]
The \emph{degree} $\deg M$ is defined as $\deg M = i + j$.
If $M = 0$, then $\deg M := -\infty$. 

For a polynomial function in the traces of matrix monomials, its degree is the total degree of the polynomial in the entries of $X$ and $Y$.
\end{definition}

We use the notation 
\begin{align} \label{notation-sim}
    f \thicksim g \text{ mod } d
\end{align}
if $f-g$ on $\cabove$ agrees with a polynomial function in traces of matrix monomial of degree $\le d$.

Given two Hamiltonian functions $F$ and $G$ on $\matnn \times \matnn$, their Poisson bracket associated with the symplectic form $\widehat{\Omega} = dX \wedge dY = \sum_{j,k} dX_{jk} \wedge dY_{kj}$ is
\begin{align} \label{poissonbraket}
    \{ F, G \} 
    =  \sum_{j,k = 1}^n \frac{\partial F}{\partial X_{jk}} \frac{\partial G}{\partial Y_{kj}} - \frac{\partial F}{\partial Y_{jk}} \frac{\partial G}{\partial X_{kj}}  
\end{align}

\begin{lemma} \label{lem:bracketformular}
    On $\cabove$ 
    \begin{align*}
        & \quad \{ \tr A^j B^k, \tr A^p B^q \}  \\
        &\thicksim (jq-kp) \tr A^{j+p-1} B^{k+q-1}   \\
        &\quad - \frac{1}{n} \left( jq \, \tr A^{j-1}B^k \cdot \tr A^p B^{q-1} - kp \, \tr A^j B^{k-1} \cdot \tr A^{p-1} B^q \right) \\
        &\quad \mod j+p+k+q-6
    \end{align*}
\end{lemma}
\begin{proof}
To compute $\{ \tr A^j B^k, \tr A^p B^q \}$ we first consider
\begin{align*}
    \frac{\partial A_{jk}}{\partial X_{lm}} = \frac{\partial}{\partial X_{lm}} \left(X_{jk} - \frac{1}{n} (\tr X)  \delta_{jk} \right) 
    = \delta_{lj} \delta_{mk} - \frac{1}{n} \delta_{lm} \delta_{jk}
\end{align*}
and
\begin{align} \label{BpqYml}
    \frac{\partial B_{pq}}{\partial Y_{ml}} = \frac{\partial}{\partial Y_{ml}} \left(Y_{pq} - \frac{1}{n} (\tr Y)  \delta_{pq} \right) 
    = \delta_{mp} \delta_{lq}  - \frac{1}{n} \delta_{ml} \delta_{pq}
\end{align}
Combining them together gives
\begin{align} \label{derivation-rule}
    \sum_{l,m} \frac{\partial A_{jk}}{\partial X_{lm}} \frac{\partial B_{pq}}{\partial Y_{ml}} &= \delta_{jq} \delta_{kp} - \frac{1}{n} \delta_{jk} \delta_{pq} - \frac{1}{n} \delta_{pq} \delta_{jk} + \frac{1}{n^2} n \delta_{jk} \delta_{pq} \\ 
    &= \delta_{jq} \delta_{kp} - \frac{1}{n} \delta_{jk} \delta_{pq} \nonumber
\end{align}
Using this identity we compute
\begin{align*} 
    & \quad \{ \tr A^j B^k, \tr A^p B^q \}  \\
    & = \sum_{l,m} \frac{\partial}{\partial X_{lm}} (\tr A^j B^k) \frac{\partial}{\partial Y_{ml}}(\tr A^p B^q) - \frac{\partial}{\partial X_{lm}} (\tr A^p B^q) \frac{\partial}{\partial Y_{ml}}(\tr A^j B^k)  \\
    & = (jq-kp) \tr A^{j+p-1} B^{k+q-1} + ( \text{terms of degree} \le j+p+k+q-6) \\
    & \quad - \frac{1}{n} \left( jq \, \tr A^{j-1}B^k \cdot \tr A^p B^{q-1} - kp \, \tr A^{p-1} B^q \cdot \tr A^j B^{k-1} \right)
\end{align*} 
where in the last step we use the fact that on $\cabove$ any trace of a matrix monomial of bidegree $(i,j)$ can be written as $\tr X^iY^j$ plus a polynomial in traces of matrix monomials in $X$ and $Y$, whose degrees are bounded by $i+j-4$, see \cite{CaloSymplo}*{Corollary 26}. 
\end{proof}

For later reference, we collect the following special cases of Lemma \ref{lem:bracketformular}:
\begin{align}
    \{ \tr A^j, \tr A^k\} &= \{ \tr B^j, \tr B^k\} =0 \nonumber \\
    \{ \tr A^j, \tr B^q \}  &=  jq \, \tr A^{j-1} B^{q-1} - \frac{jq}{n} \tr A^{j-1} \cdot \tr B^{q-1} \label{AjBk} \\
    \{ \tr A^j , \tr A^p B^q \} &\thicksim jq \, \tr A^{j+p-1} B^{q-1} - \frac{jq}{n} \tr A^{j-1} \cdot \tr A^p B^{q-1}   \\
    & \quad   \mod j+p+q-6 \nonumber \\
    \{ \tr A^j B^k , \tr A B \} &= (j-k) \tr A^j B^k  \label{ABweight}\\
    \{ \tr A^j B^k , \tr A^2 \} &\thicksim - 2k \tr A^{j+1} B^{k-1} \mod j+k-4 \\
    \{ \tr A^j, \tr B^2 \} &= 2j \tr A^{j-1}B  \label{trAjB}
\end{align}
Note that the cases with equality follow from a direct application of \eqref{derivation-rule} and the cyclicity of the trace. 

We compute another set of brackets of the form $\{\tr X, \tr A^j B^k \}$. 

\begin{lemma} \label{poissonCommute}
    As functions on the Calogero--Moser space $\camo$, $\tr X$ and $\tr Y$ Poisson commute with $\tr A^j B^k$. In particular, the flow induced by $\tr A^j B^k$ preserves the subvariety 
    \[
        \tlcamo = \{ p \in \camo : \tr X(p) = \tr Y (p) = 0 \}. 
    \]
\end{lemma}
\begin{proof}
    We compute explicitly $\{\tr X, \tr A^j B^k \}$ using \eqref{BpqYml}:
    \begin{align*}
        & \quad \{\tr X, \tr A^j B^k \} \\
        &= \sum_{l,m} \frac{\partial}{\partial X_{lm}}(\tr X) \frac{\partial}{\partial Y_{ml}} (\tr A^j B^k)
        = k \sum_{l,m} \delta_{lm} \sum_{p,q} (A^j B^{k-1})_{qp} \cdot \frac{\partial B_{pq}} {\partial Y_{ml}} \\
        &= k \sum_{l,m} \delta_{lm} \sum_{p,q} (A^j B^{k-1})_{qp} \cdot (\delta_{mp} \delta_{lq}  - \frac{1}{n} \delta_{ml} \delta_{pq}) 
        = 0 .
    \end{align*}
    By symmetry $\{\tr Y, \tr A^j B^k \}$ also vanishes. 
\end{proof}

\section{Proof of Hamiltonian Density Property for $\tlcamo$} \label{proof-HDP}
The computational approach is similar to one for the Calogero--Moser space $\camo$ in \cite{CaloSymplo}. 

\begin{theorem} \label{theorem: HDP}
    The Lie algebra generated by the restrictions to $\rabove$ of the polynomials in $\mathcal{F} = \{  \tr A^2, \tr B^2, \tr A^3, (\tr AB)^2 \}$ is the entire Lie algebra of invariant polynomials on $\rabove$.
\end{theorem}

To prove Theorem \ref{theorem: HDP}, we will prove that for all natural numbers $m \ge 1$ and $p_k, q_k \ge 0$ $(k = 1, 2, \dots, m)$ 
\begin{align} \label{PrefaceThmHDP}
    \prod_{k=1}^m \tr A^{p_k}B^{q_k} \in \mathrm{Lie}(\mathcal{F}) \tag{$*$} \quad \text{ on } \rabove.
\end{align}

We will do this by multiple inductions on $m$, on the degree $D$ of the product, and on another index that will occur in the course of the proof. The various induction proofs will be formulated as separate lemmas. When $D=0$, \eqref{PrefaceThmHDP} represents the constant function $1$, which is mapped by the isomorphism $\Theta$ in \eqref{iso-Theta} to the zero vector field, which is in the Lie algebra generated by $\Theta(\mathcal{F})$. The overall inductive assumption is that with some $D = 0,1,2, \dots$
\begin{align} \label{indhyp}
    \text{\eqref{PrefaceThmHDP} holds when } \sum_k (p_k + q_k) \le D \tag{$**$}
\end{align}

Recall the notation $f \thicksim g \mod d$ if $f-g$ on $\cabove$ agrees with a polynomial function in traces of matrix monomials of degree $\le d$. When $\deg f \le D+4$, then by \eqref{indhyp} the difference $f-g$ is in $\mathrm{Lie}(\mathcal{F})$.

The proof of Theorem \ref{theorem: HDP} starts with the generating set 
\begin{align*}
    \mathcal{F} = \{  \tr A^2, \tr B^2, \tr A^3, (\tr AB)^2 \}
\end{align*} 
of Hamiltonian functions corresponding to complete vector fields according to Lemmas \ref{lemma:completeVFs}, \ref{lem:complete-OneVariable}.

\begin{lemma} \label{lemma:completeVFs}
    Hamiltonian functions of the form $H(A)$ or $H(B)$, i.e.\ depending either only on $A$ or only on $B$, and $H = \tr AB$ induce complete vector fields.
\end{lemma}
\begin{proof}
    By the explicit expression \eqref{poissonbraket} of the Poisson bracket on $\camo$, the Hamiltonian vector field associated to a function $H$ is
    \begin{align*}
        V_H = \sum_{j,k} \frac{\partial H}{\partial Y_{jk}} \frac{\partial }{\partial X_{kj}} - \sum_{j,k} \frac{\partial H}{\partial X_{jk}} \frac{\partial }{\partial Y_{kj}}. 
    \end{align*}
    If $H$ depends only on $X$, then the flow of $V_H$ is a polynomial shear in the $Y$-direction and thus complete. Similarly, if $H$ depends only on $Y$.  Furthermore, since $A= X - \frac{\tr X}{n}I_n$ and $B= Y - \frac{\tr Y}{n}I_n$, $H(A)$ and $H(B)$ induce complete Hamiltonian vector fields on $\camo$ as well. 
    For $H= \tr AB$ we write
    \begin{align*}
        \tr AB = \tr XY - \frac{1}{n} \tr X \tr Y,
    \end{align*}
    and can read off the associated vector field:
    \begin{align*}
        \quad V_H = &
        \sum_{j \neq k} X_{kj}\frac{\partial}{\partial X_{kj}} + \sum_{j=1}^n (X_{jj}-\frac{1}{n} \tr X) \frac{\partial}{\partial X_{jj}} \\
        & - \sum_{j \neq k} Y_{kj}\frac{\partial}{\partial Y_{kj}} - \sum_{j=1} ^n (Y_{jj} -\frac{1}{n} \tr Y) \frac{\partial}{\partial Y_{jj}}.
    \end{align*}
    Its flow is complete since this vector field corresponds to a linear system of ODEs with constant coefficients. 
\end{proof}

\begin{lemma} \label{lem:complete-OneVariable}
    Let $H$ be a Hamiltonian function such that $\Theta(H)$ is complete. Let $f \colon \CC \to \CC$ be an entire function in one variable. Then $\Theta(f(H))$ is complete.  
\end{lemma}
\begin{proof}
    The Hamiltonian vector field $\Theta(f(H)) = f'(H) \Theta(H)$ is complete since $f'(H)$ is in the kernel of the complete vector field $\Theta(H)$.
\end{proof}

\begin{lemma} \label{lemma: trAj&trBk}
     For any integer $j, k \ge 2$, $\tr A^j, \tr B^k$ are in $\lie(\mathcal{F})$.
\end{lemma}
\begin{proof}
    For $k \le 3$ we only need to check $\tr B^3 \in \lie(\mathcal{F})$, which follows from \eqref{AjBk} and applying $\{ \tr B^2, \cdot \}$ to $\tr A^3$ thrice. Then we proceed by induction on $k \ge 3$. The induction step follows from
    \begin{align*}
        \{ \tr A^k , \tr B^2 \} &= 2k \, \tr A^{k-1}B \\
        \{ \tr A^2, \{ \tr A^2, (\tr AB)^2 \} \} &=  8 (\tr A^2)^2 \\
        \{ (\tr A^2)^2, \tr A^{k-2} B \} &= 2 \tr A^2 \tr A^{k-1} \\
        \{ \tr A^3, \tr A^{k-1}B \} &= 3 \, \tr A^{k+1} - \frac{3}{n} \tr A^2 \tr A^{k-1}
    \end{align*} 
    Similarly for $\tr B^{k+1}$. 
\end{proof}

\begin{lemma}
    For any integer $j, k \ge 2$, $(\tr A^j)^2, (\tr B^k)^2$ are in $\lie(\mathcal{F})$.
\end{lemma}
\begin{proof}
    By \eqref{ABweight}
    \begin{align*}
        \{ \tr A^j, (\tr AB)^2 \} &= 2 j \tr A^j \tr AB \\
        \{\tr A^j, \tr A^j \tr AB \} &= j (\tr A^j)^2
    \end{align*}
    By symmetry $(\tr B^k)^2 \in \lie(\mathcal{F})$. 
\end{proof}

\begin{lemma} \label{lemma: trXYPRODtrX}
    Assume that the induction hypothesis \eqref{indhyp} holds, and let 
    \[  p, q, m \in \mathbb{N}_0, i_1, \dots, i_m  \ge 2
    \]
    with $p+q+\sum i_k \le D+4$. Then 
    \begin{align*}
        \tr A^p B^q  \prod_{k=1}^m \tr A^{i_k} \in \lie (\mathcal{F}). 
    \end{align*}
\end{lemma}
\begin{proof}
    The special case $q = 0$ follows from 
    \begin{align}
        \{ \tr A^j, (\tr AB)^2 \} &= 2 j \tr A^j \tr AB \nonumber \\
        \{ \tr A^j \tr AB, \prod_{k=1}^m \tr A^{i_k} \} &= - \left(\sum_{k=1}^m i_k \right) \tr A^j \prod_{k=1}^m \tr A^{i_k} \label{all-p_k=0}
    \end{align}
    which holds without condition on degrees. 

    When $q > 0$ we proceed by induction on $m$.
    The base case $m=0$ follows from $\tr B^{p+q} \in \lie(\mathcal{F})$ and 
    \begin{align*}
        \{ \tr A^2, \tr B^{p+q} \} &= 2 (p+q) \tr A B^{p+q-1} \\
        \{ \tr A^2, \tr A B^{p+q-1} \} &\thicksim 2 (p+q-1) \tr A^2 B^{p+q-2} \mod p+q-4
    \end{align*}
    When $p+q -4 \le D$, the hypothesis \eqref{indhyp} implies that additional summands of lower degree are in $\lie(\mathcal{F})$. Inductively
    \begin{align} \label{pullingright}
        \{ \tr A^2, \tr A^j B^{p+q-j} \} \thicksim 2 (p+q-j) \tr A^{j+1} B^{p+q-j-1}  \mod p+q-4
    \end{align}
    we obtain any $\tr A^p B^q$ as long as $p+q \le D +4$. 
    The next case $m=1$ follows from the identity \eqref{AjBk} and the same process of taking brackets with $\tr A^2$ as in \eqref{pullingright}. 
    
    Assume the lemma holds for some $m \ge 0$, and consider a product with $(m+1)$ factors of the form $\tr A^{i}$
    \[
        f= \tr A^p B^q  \prod_{k=0}^m \tr A^{i_k}, \text{ where } p+q+\sum_{k=0}^m i_k \le D+4 
    \]
    Since $i_0 \ge 1$, we have
    \[
        p+q+1+\sum_{k=1}^m i_k \le D+4 
    \]
    Thus $\tr B^{p+q+1} \prod_{k=1}^m \tr A^{i_k} \in \lie(\mathcal{F})$ by the induction assumption on $m$. Hence
    \begin{align*}
        \lie(\mathcal{F}) &\ni \{ \tr A^{i_0+1}, \tr B^{p+q+1} \prod_{k=1}^m \tr A^{i_k} \} \\
        &= (i_0+1)(p+q+1) \tr A^{i_0}B^{p+q} \prod_{k=1}^m \tr A^{i_k} \\
        &\quad - \frac{(i_0+1)(p+q+1)}{n} \tr B^{p+q} \tr A^{i_0} \prod_{k=1}^m \tr A^{i_k}
    \end{align*}
    The first summand is in $\lie(\mathcal{F})$ by the induction assumption on $m$; a repeated application of $\{ \tr A^2, \cdot \}$ on the second summand yields $f \in \lie(\mathcal{F})$, which completes the induction step. 
\end{proof}

Next, we aim for more factors of the form $\tr A^p B^q$. 
\begin{lemma} \label{lemma: trXiatrXpbYqb}
    Assume that the induction hypothesis \eqref{indhyp} holds, and let $l \in \mathbb{N}, m \in \mathbb{N}_0$, $i_1,\dots,i_l \ge 2, p_1, \dots, p_m,$ $q_1,\dots,q_m \in \mathbb{N}_0$ with $\sum_a i_a+\sum_b (p_b + q_b) \le D+4$. Then 
    \begin{align*}
        \left( \prod_{a=1}^l \tr A^{i_a} \right) \prod_{b=1}^m \tr A^{p_b}B^{q_b} \in \lie (\mathcal{F}).
    \end{align*}
\end{lemma}
\begin{proof}
    Use induction on $m$; Lemma \ref{lemma: trXYPRODtrX} covers the base cases $m=0$ and $m=1$.
    Suppose the lemma holds for some $m \ge 1$, and consider a product with $m+1$ factors of the form $\tr A^p B^q$
    \begin{align*}
        f= \left( \prod_{a=1}^l \tr A^{i_a} \right) \prod_{b=0}^m \tr A^{p_b}B^{q_b}
    \end{align*}
    where 
    \begin{align}\label{degcondi}
        \sum_{a=1}^l i_a  + \sum_{b=0}^m (p_b + q_b) \le D+4.
    \end{align}
    
    When $q_c = 0$ for some $c \in \{0, 1, \dots, m\}$, then it follows from the induction assumption on $m$ and \eqref{degcondi}
    \begin{align*}
        f = \left(\prod_{a=1}^l \tr A^{i_a}\right) \tr A^{p_c} \prod_{\substack{b\neq c}} \tr A^{p_b}B^{q_b} \in \lie (\mathcal{F})
    \end{align*}
    It remains to consider the case $q_b \ge 1$ for all $b=0, 1, \dots, m$. By Lemma \ref{lemma: trXYPRODtrX} and the induction assumption on $m$ 
    \[
        \tr A^{p_0+1}B^{q_0}, \tr A^{p_m} B^{q_m+1}  \left(\prod_{b=1}^{m-1} \tr A^{p_b}B^{q_b} \right) \prod_{a=1}^l \tr A^{i_a} \in \lie(\mathcal{F})
    \]
    Thus

    \begin{align}
        &\lie(\mathcal{F}) \ni \{ \tr A^{p_0+1} B^{q_0}, \tr A^{p_m} B^{q_m+1}  \left(\prod_{b=1}^{m-1} \tr A^{p_b}B^{q_b} \right) \prod_{a=1}^l \tr A^{i_a} \} \nonumber \\
        \thicksim& \, ((p_0+1) (q_m+1) - q_0 p_m) \tr A^{p_0+p_m}B^{q_0+q_m} \left( \prod_{a=1}^l \tr A^{i_a} \right)  \prod_{b=1}^{m-1} \tr A^{p_b}B^{q_b} \label{line1} \\ 
        &- \frac{(p_0+1) (q_m+1)}{n} \tr A^{p_0}B^{q_0} \tr A^{p_m} B^{q_m} \left( \prod_{a=1}^l \tr A^{i_a} \right)  \prod_{b=1}^{m-1} \tr A^{p_b}B^{q_b} \label{line2} \\
        &+ \frac{q_0 p_m}{n}  \tr A^{p_0+1} B^{q_0-1} \tr A^{p_m-1}B^{q_m+1} \left( \prod_{a=1}^l \tr A^{i_a} \right)  \prod_{b=1}^{m-1} \tr A^{p_b}B^{q_b} \label{line3} \\
        &+  \tr A^{p_m} B^{q_m+1} \left( \prod_{a=1}^l \tr A^{i_a} \right) \sum_{b=1}^{m-1} ((p_0+1) q_b - q_0 p_b) \tr A^{p_0+p_b}B^{q_0+q_b-1} \prod_{c \neq b} \tr A^{p_c}B^{q_c} \label{line4} \\
        &- \frac{p_0+1}{n} \tr A^{p_m} B^{q_m+1} \left( \prod_{a=1}^l \tr A^{i_a} \right) \tr A^{p_0}B^{q_0} \sum_{b=1}^{m-1} q_b   \tr A^{p_b}B^{q_b-1} \prod_{c \neq b} \tr A^{p_c}B^{q_c} \label{line5} \\
        &+ \frac{q_0}{n} \tr A^{p_m} B^{q_m+1} \left( \prod_{a=1}^l \tr A^{i_a} \right) \tr A^{p_0+1} B^{q_0-1}\sum_{b=1}^{m-1}  p_b  \tr A^{p_b-1}B^{q_b} \prod_{c \neq b} \tr A^{p_c}B^{q_c} \label{line6} \\
        &- q_0 \tr A^{p_m} B^{q_m+1} \left(\prod_{b=1}^{m-1} \tr A^{p_b}B^{q_b} \right) \sum_{a=1}^l i_a \tr A^{p_0+i_a}B^{q_0-1} \left(\prod_{d \neq a} \tr A^{i_d}\right)  \label{line7} \\
        &+ \frac{q_0}{n} \tr A^{p_m} B^{q_m+1} \left(\prod_{b=1}^{m-1} \tr A^{p_b}B^{q_b}\right) \sum_{a=1}^l i_a  \tr A^{p_0+1} B^{q_0-1} \tr A^{i_a-1} \left(\prod_{d \neq a} \tr A^{i_d}\right)  \label{line8} \\
        &\mod  \sum_a i_a + \sum_{b=0}^m (p_b + q_b) -4  \nonumber
    \end{align}
    The first summand \eqref{line1} and each summand on line \eqref{line4} have $m$ factors $\tr A^p B^q$ and are in $\lie(\mathcal{F})$ by the induction assumption on $m$.
 
    Next, for $q_0=1$ the terms on line \eqref{line3}, line \eqref{line6}, line \eqref{line7} and line \eqref{line8} all belong to $\lie(\mathcal{F})$ by the induction assumption on $m$. Thus in this case it remains to consider line \eqref{line2} and \eqref{line5}.
    For terms on line \eqref{line5}, we use another induction on the multiindex $(q_1, \dots, q_{m-1})$. For 
    \begin{align*}
        (q_1, \dots, q_{m-1} ) = (1,  \dots, 1)
    \end{align*}
    each summand on line \eqref{line5} has $m$ factors $\tr A^p B^q$ and is in $\lie(\mathcal{F})$. Hence on line \eqref{line2}
    \begin{align*}
        \tr A^{p_0}B \cdot \tr A^{p_m} B^{q_m} \left( \prod_{a=1}^l \tr A^{i_a} \right)  \prod_{b=1}^{m-1} \tr A^{p_b}B \in \lie(\mathcal{F})
    \end{align*}
    Suppose that 
    \begin{align} \label{lexiInduct}
        \tr A^{p_0}B \cdot \tr A^{p_m} B^{q_m} \left( \prod_{a=1}^l \tr A^{i_a} \right)  \prod_{b=1}^{m-1} \tr A^{p_b}B^{q_b} \in \lie(\mathcal{F})
    \end{align}
    for all $(q_1, \dots, q_{m-1} ) \prec (s_1,  \dots, s_{m-1})$ in lexicographic order. Then it is also in $\lie(\mathcal{F})$ with $(q_1, \dots, q_{m-1} )=(s_1,  \dots, s_{m-1})$, since the terms on line \eqref{line5} are also of this form and have $(q_1, \dots, q_{m-1} ) \prec (s_1,  \dots, s_{m-1})$ as exponents in the powers of $B$. An induction on $(q_1, \dots, q_{m-1} )$ in lexicographic order shows that \eqref{lexiInduct} is valid for any $(q_1, \dots, q_{m-1} )$ satisfying \eqref{degcondi}. 

\smallskip
    Continue with our induction on $q_0$. 
    Suppose that the claim holds when one of the powers of $B$, $q_c < q_0, c \in \{ 0, 1, \dots, m \}$. Then the terms on line \eqref{line3}, line \eqref{line6}, line \eqref{line7} and line \eqref{line8} belong to $\lie(\mathcal{F})$ by this assumption. Again, it remains to consider line \eqref{line2} and \eqref{line5}.
    As before we proceed by induction on the multiindex $(q_1, \dots, q_{m-1} )$ in lexicographic order to include line \eqref{line5} in $\lie(\mathcal{F})$ and to get line \eqref{line2} $f \in \lie(\mathcal{F})$. 
    This finishes the induction step.
\end{proof}

\begin{lemma} \label{lemma: nfactors}
    Assume that the induction hypothesis \eqref{indhyp} holds, and let 
    \[  m \in \mathbb{N}, p_1, \dots, p_m, q_1,\dots,q_m \in \mathbb{N}_0
    \]
    with $\sum_k p_k + q_k \le D+4$. Then 
    \begin{align*}
          \prod_{k=1}^m \tr A^{p_k}B^{q_k} \in \lie(\mathcal{F}).
    \end{align*}
\end{lemma}
\begin{proof}
    By induction on $m$; the case $m=1$ is a special case of Lemma \ref{lemma: trXYPRODtrX}. Suppose the lemma holds for some $m \ge 1$, and consider an $(m+1)$-fold product
    \begin{align*}
        f = \prod_{k=0}^m \tr A^{p_k} B^{q_k}, \text{ where } \sum_{k=0}^m (p_k + q_k) \le D + 4
    \end{align*}
    We may assume $p_0 > 0$ (for if none of $p_k$, $k \in \{0, \dots, m\}$, is positive, it reduces to the symmetric version of \eqref{all-p_k=0}). 
    By Lemmas \ref{lemma: trXYPRODtrX} and \ref{lemma: trXiatrXpbYqb}, 
    \[
        \tr B^{q_0 +1}, \tr A^{p_0 +1} \prod_{k=1}^m \tr A^{p_k} B^{q_k} \in \lie(\mathcal{F}).
    \]
    Thus
    \begin{align*}
        &\lie(\mathcal{F}) \ni \{ \tr A^{p_0+1} \prod_{k=1}^{m} \tr A^{p_k}B^{q_k}, \tr B^{q_0+1} \}  \\
        &\thicksim \, (p_0+1)(q_0+1) \tr A^{p_0} B^{q_0} \prod_{k=1}^{m} \tr A^{p_k}B^{q_k} \\
        &\quad  - \frac{(p_0+1)(q_0+1)}{n} \tr A^{p_0} \tr B^{q_0} \prod_{k=1}^{m} \tr A^{p_k}B^{q_k}  \\
        &\quad + (q_0+1) \tr A^{p_0+1} \sum_{k=1}^{m} p_k (\tr A^{p_k-1}B^{q_0+q_k} - \frac{1}{n} \tr A^{p_k-1}B^{q_k} \tr B^{q_0} ) \prod_{a \neq k} \tr A^{p_a}B^{q_a} \\
        &\mod p_0+q_0 + \sum_k (p_k+q_k) -4
    \end{align*}
    As $p_0 > 0$ the summands on the third and the fourth line have degree $\le D+4$ and are in $\lie(\mathcal{F})$ by Lemma \ref{lemma: trXiatrXpbYqb}. Thus $f \in \lie(\mathcal{F})$. 
\end{proof}

\section*{Funding}
Rafael Andrist was supported by the European Union (ERC Advanced grant HPDR, 101053085 to Franc Forstneri\v{c}). Gaofeng Huang was partially supported by Schweizerischer Nationalfonds (SNSF) grant 200021-207335.

\section*{Conflict of Interest}
The authors have no relevant competing interest to disclose.

\end{document}